\documentclass[12pt]{amsart}

\usepackage{amscd}
\usepackage{amssymb,amsmath,amsthm, amsfonts}

\usepackage{amsrefs}
\usepackage{mathrsfs}
\usepackage{enumerate}
\usepackage[all]{xy}
\usepackage[top=1in, bottom=1in, left=1.25in, right=1.25in]{geometry}

\theoremstyle{plain}
\newtheorem{thm}{Theorem}[section]
\newtheorem{lem}[thm]{Lemma}
\newtheorem{cor}[thm]{Corollary}
\newtheorem{defn-lem}[thm]{Definition-Lemma}

\newtheorem{prop}[thm]{Proposition}

\theoremstyle{definition}
\newtheorem{defn}[thm]{Definition}
\newtheorem{ex}[thm]{Example}
\newtheorem{rem}[thm]{Remark}

\newcommand{\Z}{\mathbb{Z}}

\newcommand{\C}{\mathbb{C}}

\def\md #1#2#3#4#5 {\left(
                        \begin{matrix}
             #1 & #2 \\
             #3 & #4
                        \end{matrix}
                      \right)- #5}

\def\ma #1#2#3#4 {\left(
                        \begin{matrix}
             #1 & #2 \\
             #3 & #4
                        \end{matrix}
                      \right)}

\def\End{\operatorname{End}}

\def\Tr{\operatorname{Tr}}

\def\Span{\operatorname{span}}

\newcommand{\mc}{\mathcal}
\newcommand{\lb}{\langle}
\newcommand{\rb}{\rangle}
\newcommand{\mb}{\mathbb}
\newcommand{\mf}{\mathbf}
\newcommand{\cc}{\circ}
\begin{document}
\title [Fourier transform, Schr\"{o}dinger representation, and Heisenberg modules]
       {Fourier transform, Schr\"{o}dinger representation, and Heisenberg modules}

\begin{abstract}
We investigate and review how Fourier transform is involved in the analysis of a twisted group algebra $L^1(G, \sigma)$ for $G=\widehat{\Gamma}\times \Gamma$ and $\sigma:G\times G \to \mb{T}$ 2- cocycle where $\Gamma$ is a locally compact abelian group and $\widehat{\Gamma}$ its Pontryagin dual. By weaving the Schr\"{o}dinger representation and Fourier transform, we construct the dual equivalence bimodule of the Heisenberg bimodule generated by the dual Schr\"{o}dinger representation and observe several relations between them including the application of noncommutative solitons.    

\end{abstract}

\author{Hyun Ho \, Lee}

\address {Department of Mathematics\\
          University of Ulsan\\
         Ulsan, South Korea 44610 }
\email{hadamard@ulsan.ac.kr}


\date{}
\thanks{This research was supported by Basic Science Research Program through the National Research Foundation of Korea(NRF) funded by the Ministry of Education(NRF-2018R1D1A1A01057489).}
\maketitle
\setcounter{section}{-1}
\section{Introduction}
 
The noncommutative tori, which are most accessible and well understood, has been central in both $C\sp*$-algebra theory and noncommutative geometry.  Though it is simply defined as the universal $C\sp*$-algebra of $n$ unitary operators which  are not necessarily commutative but their commutators are fixed scalar multiples of the identity operator, the analysis on noncommutative tori depends on how we customize them among several equivalent pictures like crossed product $C\sp*$-algebra or deformation quantization of ordinary $n$-torus. In this article, we use a lattice picture originally taken by Rieffel in \cite{R:Morita}, but recently by several others in \cite{L:Gabor, L:VBoverNT, DLL:Sigma, DJLL:Sigma} which is better suited for noncommutative geometry. \\ 

The differentiable structure on them are defined by a natural ergodic action of $\mb{T}^n$ as a group of automorphisms and  projective modules or vector bundles over noncommutative tori  are constructed through Rieffel's Morita equivalence \cite{R:Morita} both of which are essential in noncommutative geometry \cite{ CO:NCG, CR:YM}. We note that the construction of projective modules over a noncommutative torus relies on a more general approach using the Schr\"{o}dinger representation on a phase space of the form $\Gamma \times \widehat{\Gamma}$ where $\Gamma$ is a locally compact abelian group and $\widehat{\Gamma}$ its Pontryagin dual. In fact, when we pass from the Schr\"{o}dinger representation of  phase space to that of lattices in phase space, the same equivalence bimodule over the algebra of observables in phase space and $\C$ serves as an equivalence modules over the algebra of observables in a lattice space and the algebra of observables in the dual lattice.\\         

On the other hand, the Schr\"{o}dinger representation on a latiice of the phase space $\Gamma\times \widehat{\Gamma}$ has been used in a time-frequency analysis, so called Gabor analysis. Let us briefly explain one of  central themes of Gabor analysis: Let $\pi$  be the Schr\"{o}dinger representation and $\Lambda$ a latiice.  They look for a generator $\xi$ in a suitable space such that the system $\{ \pi(\lambda)\xi \mid \lambda \in \Lambda\}$ becomes a frame for $L^2(\Gamma)$ or a pair $(\xi, \eta)$ such that the analysis operator using the system $\{ \pi(\lambda)\xi \mid \lambda \in \Lambda\}$ and the synthesis operator using the system  $\{ \pi(\lambda)\eta \mid \lambda \in \Lambda\}$ provides the reconstruction of a signal.\\  

Therefore through the Schr\"{o}dinger representation  Morita equivalence for noncommutative tori is linked to Gabor analysis on irrational parameter lattices.  Their essential interplay is expressed and captured by the fact that the associativity condition between two hermitian products in Morita equivalence is the fundamental identity in Gabor analysis \cite{L:Gabor, L:VBoverNT}. Besides this parallel result  it is remarkable that a key property from Gabor analysis, namely a duality principle, provides a relation between two hermitian products whose usefulness is demonstrated by the fact that Gabor frames are noncommuative solitons \cite{DLL:Sigma, Lee2:Sigma}.\\ 

Originally motivated by the question whether the Fourier transform of solitons over noncommutative tori are  solitons again, we determine the (Fourier) dual  Schr\"{o}dinger representation which means  an equivalent representation of the Schr\"{o}dinger representation under the Fourier transform and see relations of them in the central extension in Section \ref{S:Schrodinger}. Then we construct a (Fourier) dual of Moyal plane in the continuous case  and one of noncommutative torus in the discrete case, thus establish  Mortia equivalence bimodules over them using the dual representation in Section \ref{S:Heisenberg}.\\

Then using the (Fourier) dual constructions we investigate how solitons over Moyal plane and noncommutative torus behave under Fourier transform in Section \ref{S:solitons}.  To encompass higher rank projective modules, we keep our approach rather abstract, thereby we assume that the algebras are smooth as domains of derivations and the equivalence bimodule over those algebras are endowed with covariant derivatives compatible with derivations. There are faithful traces naturally defined  and when necessary we shall assume that  traces are invariant under the $\mathbb{T}^2-$action.  There are important quantities related to the traces and solitons and we observe how such quantities  are changed under Fourier transform.

\section{Sch\"{o}dinger representation and its dual representation }\label{S:Schrodinger}

Let $G$ be a group of the form $ \widehat{\Gamma} \times \Gamma$ where $\Gamma$ is a second countable locally compact abelian group and $\widehat{\Gamma}$ its Pontryagin dual.  Throughout the paper, we fix a Haar measure on $\Gamma$ and choose a Haar measure on $\widehat{\Gamma}$ so that  the Plancherel theorem holds; we define the Fourier transform of $f\in L^1(G)$ by 
\[ \mc{F}(f)(\gamma)= \widehat{f}(\gamma)=\int_{\Gamma} f(s)\overline{\gamma(s)}ds  \quad \text{for $\gamma \in \widehat{\Gamma}$}.\] Then $f$ can be recovered from $\widehat{f}$ by the inverse Fourier transform 
\[  f(s)=\mc{F}^{-1}(\widehat{f})(s)=\int_{\widehat{\Gamma}}\widehat{f}(\gamma)\gamma(s) d\gamma. \] 

We also denote $\breve{f}:=\mc{F}^{-1}(f) \in L^1(\Gamma)$ for $f \in C_0(\widehat{\Gamma})$.  Let us recall the definition of $2$-cocycle on a topological group $G$. 
\begin{defn}
We say a map $\sigma:G\times G \to \mb{T}$ is a 2-cocycle if the following hold;  for any $s,t,u \in G$
\begin{enumerate}
\item $\sigma(s,t)\sigma(st, u)=\sigma(s,tu)\sigma(t,u)$,
\item $\sigma(e,s)=\sigma(t,e)=1$
\end{enumerate}
where $e$ is the identity of $G$.
\end{defn}
On $G=\widehat{\Gamma}\times \Gamma$, there is a canonical 2-cocycle $\sigma :G\times G \to \mb{T}$ defined by $\sigma((\gamma_1, t_1), (\gamma_2, t_2))= \overline{\gamma_2(t_1)}$. Moreover, there is a canonical square integrable  $\sigma$-projective representation $\pi$ on $L^2(\Gamma)$, which is given by 
\[  \pi(\gamma, t)f(s)=\gamma(s)f(st^{-1}).\]

Indeed, for $f \in L^2(\Gamma)$
\[\begin{split}
\pi(\gamma_1, t_1) \pi(\gamma_2, t_2) f (s)&=\gamma_1(s)(\pi(\gamma_2, t_2) f)(st_1^{-1})\\
&=\gamma_1(s)\gamma_2(st_1^{-1})f(st_1^{-1}t_2^{-1})\\
&=\gamma_2(t_1^{-1}) (\gamma_1\cdot\gamma_2 )(s)f(s(t_1t_2)^{-1})\\
&=\sigma((\gamma_1, t_1), (\gamma_2, t_2))\pi((\gamma_1, t_1)\cdot(\gamma_2, t_2)) f(s).
\end{split}\]

Viewing $G$ as a phase space, $\pi$ is called \emph{the Heisenberg representation} or  \emph{the Sch\"{o}dinger representation} and in the case $\Gamma=\mb{R}$ it is called the time-frequency shift in signal analysis. We emphasize that $\pi$ is a composition of modulation operator and translation operator on $L^2(\Gamma)$ and the order is important; if we define  
\[ T_t: \xi(s) \to \xi(st^{-1}) \]
\[ M_{\gamma}: \xi(s) \to \gamma(s)\xi(s),\] then 
\begin{equation}
\pi(\gamma, t)=M_{\gamma}T_t
\end{equation}
In other words, $\pi(g)$ means that translation first and modulation second corresponding to $(\gamma, t)$.
It follows from the definition of $\pi$ that 
\begin{equation}
\pi(g_1)\pi(g_2)=\sigma(g_1, g_2)\overline{\sigma(g_2, g_1)}\pi(g_2)\pi(g_1) \quad \text{for} \,g_1,g_2 \in G
\end{equation}
Thus it is natural to define an anti-symmetrized 2-cocycle $\sigma_{symp}(g_1,g_2):=\sigma(g_1, g_2)\overline{\sigma(g_2, g_1)}.$ For us, it is worth to single out the second part of $\sigma_{symp}$. 
\begin{prop}
Let us define $\sigma^*(g_1, g_2):=\overline{\sigma(g_2, g_1)}$ for $G=\widehat{\Gamma}\times \Gamma$. Then it is a 2-cocycle again. 
\end{prop} 
Now for $g=(\gamma, t)$ \[\begin{split}
\lb \pi(g)\xi, \eta \rb_{L^2(\Gamma)} &=\int_{\Gamma} \pi(\gamma, t)\xi(s)\overline{\eta(s)}\,ds \\
&=\int_{\Gamma}\gamma(s)\xi(st^{-1})\overline{\eta(s)}\,ds \\
&=\int_{\Gamma}\gamma(st)\xi(s)\overline{\eta(st)}\,ds\\
&=\int_{\Gamma}\xi(s)\gamma(t)\overline{\gamma^{-1}(s)}\overline{\eta(st)}\,ds \\
&=\int_{\Gamma}\xi(s) \overline{\sigma(g,g)} \overline{\pi(\gamma^{-1}, t^{-1})\eta(s)},ds\\
&=\lb \xi, \sigma(g,g)\pi(g^{-1})\eta \rb_{L^2(\Gamma)}
\end{split}\]
which implies that 
\begin{equation}\label{E:adjoint}
\pi(g)^*=\sigma(g,g)\pi(g^{-1})
\end{equation}

Then we define another projective representation of $G$ on $L^2(\Gamma)$ as follows; 
\[ \pi^{\circ}(\gamma, t):=T_t M_{\gamma}. \]
Because the order is changed, it is called a frequency-time shift in signal analysis.
\begin{prop}
$\pi^{\circ}$ is a square-integrable $\sigma^*$-projective representation.
\end{prop}
\begin{proof}
We note that 
\begin{equation}
M_{\gamma}T_t=  \gamma(t)T_tM_{\gamma}.
\end{equation}
which means that 
\[ \pi(g)=\overline{\sigma(g,g)}\pi^{\circ}(g).\]
Hence for $g_1=(\gamma_1,t_1), g_2=(\gamma_2,t_2)$
\[ \begin{split}
\pi^{\circ}(g_1)\pi^{\circ}(g_2)&= \sigma(g_1,g_1)\pi(g_1)\sigma(g_2.g_2)\pi(g_2)\\
&=\sigma(g_1.g_1)\sigma(g_2,g_2)\sigma(g_1.g_2)\pi(g_1g_2)\\
&=\overline{\gamma_1(t_1)} \overline{\gamma_2(t_2)} \overline{\gamma_2(t_1)}\pi(g_1g_2)\\
&=\gamma_1(t_2) \overline{\gamma_1\gamma_2(t_1t_2)} \pi(g_1g_2)\\
&=\sigma^*(g_1,g_2)\sigma(g_1g_2, g_1g_2)\pi(g_1g_2)\\
&=\sigma^*(g_1,g_2)\pi^{\circ}(g_1g_2).
\end{split}\]
\end{proof}
We also note that 
\begin{equation}
\pi^{\circ}(g^{-1})=\pi(g)^*.
\end{equation} since $\sigma(g, g)=\sigma(g^{-1}, g^{-1})$.

In general, when we are given a 2-cocycle $\sigma$ on a topological group $G$ and a $\sigma$-projective representation $\pi$ on a complex Hilbert space we can lift $\pi$ of $G$ to a linear representation of $G\times_{\sigma} \mb{T}$ which is a central extension of $G$. Let us describe a group $G\times_{\sigma} \mb{T}$ by explaining the group law; 
\[(g_1,z_1) \cdot (g_2.z_2):=(g_1g_2,\sigma(g_1,g_2)z_1z_2).\]
It is a good exercise to check that the associativity holds since $\sigma$ is a 2-cocycle and the representation $\widetilde{\pi}(g,z):=\pi(g)S_z$, where $S_z$ is the scalar multiplication of $z$, is an ordinary representation of $G\times_{\sigma}\mb{T}$. We denote by $\widetilde{\pi}$ the representation of  $G\times_{\sigma} \mb{T}$ associated with a $\sigma$-projective representation $\pi$ of $G$.

Let $\widehat{G}=\Gamma \times \widehat{\Gamma}$ and define a $2$-cocycle $\sigma'((t_1,\gamma_1),(t_2,\gamma_2))=\overline{\gamma_1(t_2)}$. Using the Pontryagin duality $\widehat{\widehat{\Gamma}} \cong \Gamma$, we can view $t\in \Gamma$ as a character from $\widehat{\Gamma} \to \mb{T}$ so that $\pi(t,\gamma)=M_tT_{\gamma}$ and $\pi^{\circ}(t,\gamma) = T_{\gamma}M_t$ are  $\sigma$ and $(\sigma')^*$-projective representation of $\widehat{G}$ on $L^2(\widehat{\Gamma})$ respectively.  

It is interesting to see that Fourier transform $\mc{F}: L^2(\Gamma) \to L^2(\widehat{\Gamma})$ or $\wedge: L^2(\Gamma) \to L^2(\widehat{\Gamma})$ is an intertwining operator between two representations on $\Gamma\times \widehat{\Gamma} \times_{\sigma'}\mb{T}$ and $\widehat{\Gamma}\times \Gamma \times_{\sigma} \mb{T}$ respectively; more precisely  we define an isomorphism $r$ from $\widehat{\Gamma}\times \Gamma \times_{\sigma} \mb{T}$  to $\Gamma\times \widehat{\Gamma} \times_{\sigma'}\mb{T}$ as follows;
\[ r(\gamma,t, z)=(t^{-1}, \gamma, \gamma(t)z)\] 
We need to verify that $r$ is a homomorphism. 
\[\begin{split}
r((\gamma_1, t_1,z_1)\cdot(\gamma_2, t_2, z_2))&=r((\gamma_1\gamma_2, t_1t_2, \overline{\gamma_2(t_1)}z_1z_2))\\
&=(t_1^{-1}t_2^{-1}, \gamma_1\gamma_2, \gamma_1(t_1t_2)\gamma_2(t_1t_2)\overline{\gamma_2(t_1)}z_1z_2)\\
&=(t_1^{-1}t_2^{-1}, \gamma_1\gamma_2, \gamma_1(t_1t_2)\gamma_2(t_2)z_1z_2).\end{split}\]
On the other hand, 
\[\begin{split}
r(\gamma_1, t_1, z_1)\cdot r(\gamma_2, t_2, z_2)&=(t_1^{-1}, \gamma_1, \gamma_1(t_1)z_1)\cdot (t_2^{-1}, \gamma_2, \gamma_2(t_2)z_2)\\
&=(t_1^{-1}t_2^{-1}, \gamma_1\gamma_2, \sigma'((t_1^{-1}, \gamma_1), (t_2^{-1},\gamma_2))\gamma_1(t_1)\gamma_2(t_2)z_1z_2)\\
&=(t_1^{-1}t_2^{-1}, \gamma_1\gamma_2, \gamma_1(t_2)\gamma_1(t_1)\gamma_2(t_2)z_1z_2 )
\end{split}\]

\begin{prop}
Let $\iota: G \to \widehat{G}$ be a map defined by $\iota(\gamma, t)=(t^{-1}, \gamma)$.
$\pi^{\circ}\circ \iota: G \to B(L^2(\widehat{\Gamma}))$ is a $\sigma$-projective representation of $G$ and $\widetilde{\pi^{\circ}\circ \iota}= \widetilde{\pi}\circ r$, i.e., the lift of $\pi^{\circ}\circ \iota$ is $\widetilde{\pi}\circ r$. 
\end{prop}
\begin{proof}
For $\xi \in L^2(\widehat{\Gamma})$
\[ \begin{split}
(\widetilde{\pi}\circ r) (\gamma,t, z) \xi (\delta)&=\widetilde{\pi}(t^{-1}, \gamma, \gamma(t)z) \xi (\delta)\\
&=\overline{\delta(t)}\gamma(t)z\xi(\delta\gamma^{-1})\\
&=z \delta\gamma^{-1}(t^{-1})\xi(\delta\gamma^{-1})\\
&=S_z\pi^{\circ}(t^{-1},\gamma)\xi(\delta)\\
&=S_z\pi^{\circ}\circ \iota (\gamma, t) \xi (\delta)
\end{split}
\]
\end{proof}
Combining the above result with the uniqueness of Heisenberg commutation relations and its extension to a $\sigma$-twisted group  the following is expected. 
\begin{thm}
The Fourier transform $\mc{F}$ is the intertwining unitary transformation between $\pi:G \to B(L^2(\Gamma))$ and $\pi^{\circ}\circ \iota:G \to B(L^2(\widehat{\Gamma}))$ and lifted between $\widetilde{\pi}$ and $\widetilde{\pi^{\circ}\circ \iota}$. The following diagram summarizes our statement.\\
\xymatrix{
& L^2(\Gamma) \ar@{-}[rr]^{\wedge} \ar@{=} '[d] [dd]
&& L^2(\widehat{\Gamma}) \ar@{=}[dd]\\ 
L^2(\Gamma) \ar@{-}[ur]^{\widetilde{\pi}}
\ar@{-}[rr]^{\wedge}\ar@{=}[dd] && L^2(\widehat{\Gamma})\ar@{-}[ur]^{\widetilde{\pi}\circ r}\ar@{=}[dd]\\
       & L^2(\Gamma)\ar@{-}'[r] [rr]^{\wedge} &&  L^2(\widehat{\Gamma})\\
L^2(\Gamma) \ar@{-}[rr]^{\wedge} \ar@{-}[ur] ^{\pi}
&& L^2(\widehat{\Gamma}) \ar@{-}[ur]^{\pi^{\circ} \iota}}\\
Figure. 1.

\end{thm}
\begin{proof}
Let us show $\wedge \circ \widetilde{\pi}= \widetilde{\pi}\circ r \circ \wedge $ first; for $\xi \in L^2(\Gamma) $ and $(\gamma, t, z)\in \widehat{\Gamma}\times \Gamma\times \mb{T}$
\[\begin{split}
\mc{F}(\widetilde{\pi}(\gamma, t, z)\xi)(\delta)&=\int_{\Gamma}z \gamma(s)\xi(st^{-1})\overline{\delta(s)}\,ds \\
&= z\int_{\Gamma}\gamma(st)\xi(s)\overline{\delta(st)}\,ds \\
&=z\gamma(t)\overline{\delta(t)}\int_{\Gamma}\xi(s)\overline{\delta\gamma^{-1}(s)}ds\\
&=z\gamma(t)\overline{\delta(t)}\mc{F}(\xi)(\delta\gamma^{-1})\\
&=\widetilde{\pi}(t^{-1},\gamma, \gamma(t)z) \mc{F}(\xi)(\delta)\\
&=\widetilde{\pi}\circ r (\gamma, t, z) \mc{F}(\xi)(\delta). 
\end{split}
\]
Next, we show that 
\[\begin{split}
\pi^{\circ}\circ \iota (\gamma, t) \mc{F}(\xi)(\delta)&=\pi^{\circ}(t^{-1}, \gamma)\mc{F}(\xi)(\delta)\\
&=t^{-1}(\delta\gamma^{-1})\mc{F}(\xi)(\delta\gamma^{-1})\\
&=\delta\gamma^{-1}(t^{-1})\int_{\Gamma}\xi(s)\overline{\delta\gamma^{-1}(s)}\,ds\\
&= \delta\gamma^{-1}(t^{-1})\int_{\Gamma}\xi(st^{-1})\overline{\delta\gamma^{-1}(st^{-1})}\,ds\\
&=\int_{\Gamma}\xi(st^{-1})\overline{\delta\gamma^{-1}(s)}\,ds\\
&=\int_{\Gamma} \pi(\gamma, t)\xi(s)\overline{\delta(s)}\,ds\\
&=\mc{F}( \pi(\gamma,t) \xi ) (\delta).
\end{split}
\]
\end{proof}
From the above observation we shall call $\pi^{\circ}\circ \iota$ as the dual Schr\"{o}dinger representation of $G$. 
\begin{ex}\cite{Howe}
Let us consider $\Gamma=\mb{R}^n$. For $(\mf{y}, \mf{x}) \in G=\widehat{\Gamma} \times \Gamma$, the Schr\"{o}dinger representation is given by $\pi(\mf{y}, \mf{x})\xi(\mf{x}')=e^{2\pi i \lb\mf{y}, \mf{x}'\rb}\xi(\mf{x}'-\mf{x})$ for $\xi \in L^2(\mb{R}^n)$. In addition, if a 2- cocycle $\sigma:G\times G \to \mb{T}$ is defined as $\sigma((\mf{y_1}, \mf{x_1}),(\mf{y_2},\mf{x_2}))=e^{-2\pi i \lb \mf{x_1},\mf{y_2}\rb}$, then $\pi$ is a $\sigma$-projective representation.  The central extension of $G$ with respect to $\sigma$ or $G\times_{\sigma}\mb{T}$ is called a (reduced) Heisenberg group of order $n$ and the lifted representation of $\pi$ is $\widetilde{\pi}(\mf{y},\mf{x},z)= S_z \pi(\mf{y}, \mf{x})$ where $S_z$ is the scalar multiplcation by $z$ for $(\mf{y}, \mf{x}, z) \in G \times_{\sigma} \mb{T}$. In this case, the map $r(\mf{y}, \mf{x}, z)=(-\mf{x},\mf{y}, e^{2 \pi i \lb \mf{x}, \mf{y} \rb}z)$ is an automorphism of $G\times_{\sigma} \mb{T}$. Moreover, $\pi^{\circ}\circ \iota (\mf{y}, \mf{x}) \xi(\mf{x'})=\pi^{\circ}(-\mf{x}, \mf{y})\xi(\mf{x}')=e^{-2 \pi i \lb \mf{x}, \mf{x}'-\mf{y} \rb}\xi(\mf{x}'-\mf{y}).$ Then for $\xi \in L^2(\mb{R}^n)$ we have 
\[ \begin{split}
\mc{F}\circ \pi(\mf{y},\mf{x})(\xi)(\mf{y}')&= \int_{\mb{R}^n}  e^{2\pi i \lb\mf{y}, \mf{x}'\rb} e^{-2 \pi i \lb \mf{x}', \mf{y}' \rb} \xi(\mf{x}'-\mf{x})d\mf{x}' \\
&= \int_{\mb{R}^n} e^{2 \pi i \lb \mf{x}', \mf{y}-\mf{y}' \rb} \xi(\mf{x}'-\mf{x}) d\mf{x}'\\
&=\int_{\mb{R}^n} e^{-2 \pi i \lb \mf{x}'+\mf{x}, \mf{y}'-\mf{y} \rb} \xi(\mf{x}') d\mf{x}'\\
&= e^{-2 \pi i \lb \mf{x}, \mf{y}'-\mf{y} \rb} \int_{\mb{R}^n} e^{-2 \pi i \lb \mf{x}', \mf{y}'-\mf{y} \rb} \xi(\mf{x}') d\mf{x}'\\
&=e^{-2 \pi i \lb \mf{x}, \mf{y}'-\mf{y} \rb} \widehat{\xi}(\mf{y}'-\mf{y})\\
&=\pi^{\circ}(-\mf{x}, \mf{y})(\mc{F}(\xi))(\mf{y}').
\end{split}
\]
\end{ex}

Let us consider a twisted Banach algebra and its quantizations as a subalgebra of $B(H)$  where $H$ is an infinite dimensional Hilbert space. Following a common notation for an abelian group we now use addition instead of multiplication. Consider $L^1(G)$ for a locally compact abelian group, for for $a, b \in L^1(G)$ we define a twisted convolutions using $2$-cocycle $\sigma$ by 
\begin{equation}\label{E:twistedconv}
(a \natural b)(g)=\int_{G} a(g')b(g-g')\sigma(g',g-g') dg' 
\end{equation}
and twisted involution of $a \in L^1(G)$ as 
\begin{equation}\label{E:twistedinvo}
a^*(g)=\sigma(g,g)\overline{a(-g)}.
\end{equation}
The associated twisted group algebra is denoted by $L^1(G, \sigma)$. If there is a faithful $\sigma$-projective representation $\rho$, then we can quantize $L^1(G, \sigma)$ using the integral representation $A=\rho(a)=\int_{G} a(g)\rho(g) dg$.  The $C\sp*$-algebra generated by $\{\pi(a)\mid a \in L^1(G)\}$ is denoted by $C^*(G, \rho, \sigma) \subset B(H)$.  It is easily checked that 
\[\begin{aligned}
\rho(a)\rho(b)&=\rho(a \natural b),\\
(\rho(a))^*&= \rho(a^*).
\end{aligned}\]

Now consider the case $G=\widehat{\Gamma} \times \Gamma$.  We have two faithful representations $\pi(\gamma, t)=M_{\gamma}T_t$ on $L^2(\Gamma)$ and $\pi^{\circ}\iota (\gamma, t)=T_{\gamma}M_{-t}$ on $L^2(\widehat{\Gamma})$.  Thus we can think of  two different quantizations of $L^1(G, \sigma)$ as $C\sp*(G, \pi, \sigma) \subset B(L^2(\Gamma))$ and $C\sp*(G, \pi^{\circ}\iota, \sigma) \subset B(L^2(\widehat{\Gamma}))$. 
Let us compare the matrix coefficients of $\pi$ and $\pi^{\circ} \iota$ which are defined for $\xi, \eta \in L^2(\Gamma)$ by 
\[\lb \xi, \pi(g)\eta \rb=\int_{\Gamma} \xi(s)\bar {\eta}(s-t)\overline{\gamma(s)} ds, \]
for $\xi', \eta' \in L^2(\widehat{G})$ by
\[ \lb \xi', \pi^{\circ}\iota(g)\eta' \rb=\int_{\widehat{\Gamma}} \xi'(\delta)\bar {\eta'}(\delta-\gamma)\delta(t)\overline{\gamma(t)} d\delta\] where $g=(\gamma, t)$. 
\begin{prop}
For $\xi, \eta \in L^2(\Gamma)$ we have 
\[ \lb \xi, \pi(\gamma, t)\eta \rb =\lb \widehat{\xi}, \pi^{\circ}\iota(\gamma,t)\widehat{\eta}\rb. \]
\end{prop}
\begin{proof}
By Theorem, \[  \begin{split} 
\lb \widehat{\xi}, \pi^{\circ}\iota(\gamma,t)\widehat{\eta}\rb&= \lb \mc{F}{\xi}, \mc{F}(\pi(\gamma,t)\eta)\rb\\
&= \lb \xi, \pi(\gamma, t)\eta \rb 
\end{split}\]   
\end{proof}
In signal analysis $V_{\eta}\xi: (\gamma. t) \mapsto \lb \xi, \pi(\gamma, t)\eta \rb$ is called the short-time Fourier transform and $\eta$ plays a role of the window function. It is clear that $V_{\eta}\xi \in L^2(G)$ for $\xi, \eta \in L^2(\Gamma)$. Similarly, for $\xi', \eta' \in L^2(\widehat{\Gamma})$ we define $V^{\circ}_{\eta'}\xi': (\gamma, t) \to \lb\xi'. \pi^{\circ}\iota(\gamma, t)\eta' \rb$ as an element of $L^2(G)$ and call it the dual short-time Fourier transform.

\begin{lem}(Moyal identity)
For $\eta, \xi, \phi, \psi \in L^2(\Gamma)$ we have 
\[ \lb V_{\eta}\xi, V_{\psi}\phi \rb_{L^2(G)}=\lb \xi,\phi \rb \lb \psi, \eta \rb.\]
\end{lem}

\begin{prop}
For $\eta', \xi', \phi', \psi' \in L^2(\widehat{\Gamma})$ it holds that \[ \lb V^{\circ}_{\eta'}\xi', V^{\circ}_{\psi'}\phi' \rb_{L^2(G)}=\lb \xi',\phi' \rb \lb \psi', \eta' \rb.\]
\end{prop} 
\begin{proof}
We note that $\breve{\eta'} \in L^2(\Gamma)$ is such that $\mc{F}(\breve{\eta'})=\eta'$, or the inverse Fourier transform of $\eta'$. Then using Moyal identity and Theorem 
\[\begin{split}
\lb V^{\circ}_{\eta'}\xi', V^{\circ}_{\psi'}\phi' \rb_{L^2(G)} &=  \lb V_{\breve{\eta'}}\breve{\xi'}, V_{\breve{\psi'}}\breve{\phi'} \rb_{L^2(G)}\\
&=\lb \breve{\xi'},\breve{\phi'} \rb \lb \breve{\psi'}, \breve{\eta'} \rb\\
&=\lb \xi',\phi' \rb \lb \psi', \eta' \rb.
\end{split}
\]
\end{proof} 
One of consequences is that we can view $V^{\circ}_{\eta'}: L^2(\widehat{\Gamma}) \to L^2(G)$ as an isometry like $V_{\eta}$.
\begin{cor}
If we fix $\eta'$ such that $\| \eta'\|^2=1$ in $L^2(\widehat{\Gamma})$, then 
 \[ \|V^{\circ}_{\eta'}(\xi') \|^2_{L^2(G)}= \| \xi'\|_{L^2(\widehat{\Gamma})}^2 \]
 for $\xi' \in L^2(\widehat{\Gamma})$.
\end{cor}
If we consider the dual generating system $\{ \pi^{\circ}\iota(g)\eta' \mid g\in G \}$ for $\eta' \in L^2(\widehat{\Gamma})$ together with $\psi'$ such that $\lb \psi', \eta' \rb \ne 0$ we can recover $\xi' \in L^2(\widehat{\Gamma})$. 
\begin{cor}
Let $\eta'$ and $\psi'$ be in $L^2(\widehat{\Gamma})$ such that $\lb \psi', \eta' \rb \ne 0$. Then for any 
$\xi' \in L^2(\widehat{\Gamma})$
\[\xi'= \lb \psi', \eta' \rb^{-1}\int_G \lb \xi', \pi^{\circ}\iota(g) \eta' \rb \pi^{\circ}\iota(g)\psi' d\mu_G \]
\end{cor}
\begin{lem}
\[(\pi^{\circ}\iota (g))^*=\sigma(g,g)\pi^{\circ} \iota (g^{-1})\]
\end{lem}
\begin{proof}
On $\widehat{G}$ equipped with the 2-cocycle $\sigma'$, $\pi^{\circ}$ and $\pi$ satisfy
\begin{equation}
\pi(t, \gamma)^{*}=\sigma'((t,\gamma),(t,\gamma))\pi((t^{-1}, \gamma^{-1}))
\end{equation} 
\begin{equation}
\pi^{\circ}(t^{-1}, \gamma^{-1})=(\pi(t, \gamma))^*
\end{equation}
Therefore, \[ 
\begin{split}
(\pi^{\circ}(\iota(g)))^*&= \pi (\iota(g)^{-1})\\
&= \overline{\sigma'(\iota(g), \iota(g))}\pi(\iota(g))^*\\
&=\overline{\sigma'(\iota(g), \iota(g))}\pi^{\circ}(\iota(g^{-1}))\\
&=\sigma(g,g)\pi^{\circ}\iota(g^{-1}).
\end{split}
\]
\end{proof}
\begin{prop}\label{P:productSSFT}
Let $\xi_1, \xi_2, \eta_1, \eta_2$ be the elements in $L^2(\widehat{\Gamma})$such that  both $V^{\circ}_{\eta_2}\xi_2$ and $V^{\circ}_{\eta_1}\xi_1$ are in $L^1(G, \sigma)$. Then we have 
\begin{equation}
V^{\circ}_{\eta_2}\xi_2 \natural V^{\circ}_{\eta_1}\xi_1=\lb \xi_1, \eta_2\rb V^{\circ}_{\eta_1} \xi_2.
\end{equation}
\begin{proof}
First we note that 
\[ \begin{split}
\sigma((\gamma', t'), (\gamma',t'))\sigma((\gamma', t')^{-1}, (\gamma, t))&= \overline{\gamma'(t')}\gamma(t')\\
&=\gamma\gamma'^{-1}(t')\\
&=\overline{\sigma((\gamma', t'), (\gamma, t)(\gamma', t')^{-1})}\\
\end{split}
\]
Or we can rewrite it using addition  for $g, g' \in G=\widehat{\Gamma}\times \Gamma$
\begin{equation}\label{E:cocycl}
\sigma(g',g')\sigma(-g, g')=\sigma(g-g', g')
\end{equation} 
Then 
\[ 
\begin{split}
V^{\circ}_{\eta_2}\xi_2 \natural V^{\circ}_{\eta_1}\xi_1(g)&=\int_{G} V^{\circ}_{\eta_2}(g')V^{\circ}_{\eta_1}\xi_1(g-g')\sigma(g', g-g')d\mu_G\\
&=\int_{G}\lb \xi_2, \pi^{\circ}\iota(g')\eta_2\rb \lb \xi_1, \pi^{\circ}\iota(g-g')\eta_1\rb \sigma(g', g-g') d\mu_G\\
&=\int_G \lb \xi_2, \pi^{\circ}\iota(g')\eta_2\rb \lb \xi_1, \sigma(g-g', g')\pi^{\circ}\iota(g-g')\eta_1\rb d\mu_G\\
&=\int_G \lb \xi_2, \pi^{\circ}\iota(g')\eta_2\rb  \lb \xi_1, \sigma(g',g')\pi^{\circ}\iota(-g')\pi^{\circ}\iota(g)\eta_1 \rb d\mu_G\\
&=\int_G \lb \xi_2, \pi^{\circ}\iota(g')\eta_2\rb  \lb \xi_1, \pi^{\circ}\iota (g')^*\pi^{\circ}\iota(g)\eta_1 \rb d\mu_G\\
&=\int_G \lb \xi_2, \pi^{\circ}\iota(g')\eta_2\rb  \lb  \pi^{\circ}\iota (g')\xi_1, \pi^{\circ}\iota(g)\eta_1 \rb d\mu_G\\
&=\lb V^{\circ}_{\eta_2} \xi_2, V^{\circ}_{\xi_1}\pi^{\circ} \iota (g)(\eta_1) \rb_{L^2(G)}\\
&=\lb \xi_2, \pi^{\circ}\iota(g) \eta_1 \rb \lb \xi_1, \eta_2 \rb\\
&=\lb \xi_1, \eta_2 \rb V^{\circ}_{\eta_1} \xi_2.
\end{split}
\]
\end{proof}
\end{prop}
\begin{cor}
Let $\xi \in L^2( \widehat{\Gamma})$ such that $\| \xi\|_2=1$ and $V^{\circ}_{\xi} \xi \in L^1(G,\sigma)$. Then 
$V^{\circ}_{\xi}\xi$ is a projector in $L^1(G,\sigma)$. 
\end{cor}
\begin{proof}
By (\ref{E:twistedinvo}),  for $g\in G$\[ 
\begin{split}
(V^{\circ}_{\eta} \xi)^*(g)&=\sigma(g,g)\overline{V^{\circ}_{\eta}\xi(-g)}\\
&=\overline{\lb \xi, \sigma(g,g)\pi^{\circ}\iota(-g)\eta\rb}\\
&=\overline{ \lb \xi, \pi^{\circ}\iota(g)^*\eta \rb }\\
&=\lb \eta, \pi^{\circ}\iota(g)\xi \rb \\
&=V^{\circ}_{\xi} \eta(g).
\end{split}\]
Thus, $(V^{\circ}_{\xi}\xi)^*=V^{\circ}_{\xi}\xi$. Moreover, by Proposition \ref{P:productSSFT} it follows that 
\[ V^{\circ}_{\xi}\xi \natural V^{\circ}_{\xi}\xi= V^{\circ}_{\xi}\xi.\]
\end{proof}

\section{Heisenberg modules and some geometric observations}\label{S:Heisenberg}
In this section, we only consider the group $\Gamma$ is elementary in the sense of Bruhat or has a differential structure so that we can think of the Schwarz space $\mc{E}=\mc{S}(\Gamma)$.  Rather we could be more flexible when we consider the Feichtinger algebra $S_0(\Gamma)$ and $S_0(\widehat{\Gamma}\times \Gamma)$ \cite{J:Segal}. Though we could prove different versions of Theorem \ref{T:Morita1} \ref{T:Morita2} \ref{T:Morita3} using the Feichtinger algebra and noncommutative Wiener algebras, we refer the reader to see \cite{L:VBoverNT} for more general constructions.\\   

Let  $\mc{A}$ be the algebra of all operators of $\pi(a)$ for $a \in \mc{S}(G)$. Then we can define a left action of a smooth subalgebra  $\mc{S}(G)$ of $L^1(G, \sigma)$  on $\mc{S}(\Gamma)$ via 
\begin{equation}\label{E:action}
a \cdot \xi = \pi(a) \xi = \int_{G} a(\gamma, t)\pi(\gamma, t)\xi\, d\mu_{G}
 \end{equation}  for $a\in \mc{S}(G)$ and $\xi \in \mc{S}(\Gamma)$. In 
ddition, we can define a $\mc{A}$-valued hermitian product; 
\begin{equation}\label{E:innerproduct}
_{\mc{A}}\lb \xi, \eta \rb= \int_{G}\lb \xi, \pi(\gamma, t)\eta \rb \pi(\gamma, t) \,d\mu_{G}  
\end{equation}
for $\xi, \eta \in \mc{S}(\Gamma)$ since $V_{\eta} \xi \in \mc{S}(G)$.  In this way the space $\mc{S}(\Gamma)$ becomes a left $\mc{A}$-module.   

Similarly, we let $\mc{A}^{\circ}$ be the algebra of all operators of $\pi^{\circ}\iota(a)$ for $a \in \mc{S}(G)$ and define a left action of a smooth subalgebra $\mc{S}(G)$ of $L^1(G, \sigma)$ via 
\begin{equation} 
a \cdot \xi' = \pi^{\circ}\iota(a) \xi' = \int_{G} a(\gamma, t)\pi^{\circ}\iota(\gamma, t)\xi'\, d\mu_{G} 
\end{equation}  for $a\in \mc{S}(G)$ and $\xi' \in \mc{S}(\widehat{\Gamma})$. In 
ddition, we can define a $\mc{A}^{\circ}$-valued hermitian product; 
\begin{equation}
_{\mc{A}^{\circ}}\lb \xi', \eta' \rb= \int_{G}\lb \xi', \pi^{\circ}\iota(\gamma, t)\eta' \rb \pi^{\circ}\iota(\gamma, t) \,d\mu_{G}  
\end{equation}
for $\xi', \eta' \in \mc{S}(\widehat{\Gamma})$ since $V^{\circ}_{\eta'} \xi' \in \mc{S}(G)$.

\begin{prop}
For $\xi, \eta, \phi \in L^2(\Gamma)$
\[ \mc{F}(_{\mc{A}}\lb \xi, \eta \rb \cdot \psi)=_{\mc{A}^{\circ}}\lb \widehat{\xi},\widehat{\eta}\rb \cdot \widehat{\psi}\]
\end{prop}
\begin{proof}
Note that $_{\mc{A}}\lb \xi, \eta \rb \cdot \psi : s \mapsto \int_{G} \lb \xi, \pi(g)\eta \rb \pi(g)\psi(s)d\mu_G $. Then using Fubini Theorem 
\[ \begin{split}
\mc{F}(_{\mc{A}}\lb \xi, \eta \rb \cdot \psi)& =  \int_{G} \lb \xi, \pi(g)\eta \rb \mc{F}(\pi(g)\psi )d\mu_G \\
&= \int_{G} \lb \widehat{\xi}, \pi^{\circ}\iota (g) \widehat{\eta} \rb \pi^{\circ}\iota(g)\widehat{\psi} d\mu_G.
\end{split}
\]
\end{proof}
The following is well known due to Rieffel but explicit arguments could be extracted from \cite{DLL:Sigma}. 
\begin{thm}\label{T:Morita1}
$\mc{S}(\Gamma)$ is an equivalence bimodule between $\mc{A}$ and $\C$ with respect to (\ref{E:action}) and (\ref{E:innerproduct}). 
\end{thm} 
Using Fourier transform we have a parallel result. Since the idea of proof is almost same with one for Theorem \ref{T:Morita1}, we provide the proof for the following theorem.    
\begin{thm}\label{T:Morita2}
The space $\mc{S}(\widehat{\Gamma})$ is an equivalence bimodule between $\mc{A}^{\circ}$ and $\C$ with respect to the actions:
\[ K\cdot \xi'=  \int_{G} k(\gamma, t)\pi^{\circ}\iota(\gamma, t)\xi'\, d\mu_{G} , \quad \xi' \cdot \lambda=\xi' \bar{\lambda}
 \] for $\xi' \in \mc{S}(\widehat{\Gamma})$ and $k \in \mc{S}(G)$, $\lambda \in \C$; and $\mc{A}^{\circ}$ and $\C$-valued hermitian products: 
 \[_{\mc{A}^{\circ}}\lb \xi', \eta' \rb= \int_G V^{\circ}_{\eta'} \xi' \pi^{\circ}\iota(g)\,d\mu_G,   \quad            \lb \xi', \eta' \rb_{\C}= \lb \eta', \xi' \rb\]
\end{thm}
\begin{proof}
It is important to observe  the associativity condition 
\begin{equation}
\lb \xi', \eta' \rb \cdot \psi'= \xi' \cdot \lb \eta', \psi '\rb.
\end{equation}
Taking the scalar product with an element $\phi'$ in $L^2(\widehat{\Gamma})$, this is equivalent to the identity
\[\begin{split}
\lb  \lb  \xi', \eta' \rb \cdot \psi', \phi' \rb&=\int_G V^{\circ}_{\eta'}\xi' (g)\lb \pi^{\circ}\iota (g)\psi', \phi' \rb d\mu_G\\
&=\int_G V^{\circ}_{\eta'}\xi' (g)\overline{\lb \phi', \pi^{\circ}\iota (g)\psi' \rb} d\mu_G\\
&=\lb V^{\circ}_{\eta'}\xi', V^{\circ}_{\psi'}\phi' \rb_{L^2(G)}\\
&= \lb \xi', \phi'\rb \overline{\lb \eta', \psi' \rb}\\
&=\lb \xi'\cdot \lb \eta', \psi' \rb, \phi' \rb
\end{split}
\]
It follows that $\lb \lb \eta' ,\eta' \rb \cdot \psi', \psi' \rb= \lb \eta', \psi' \rb \overline{\lb \eta', \psi' \rb} \ge 0$, so  $\lb \eta' ,\eta' \rb \ge 0 $ in $\mc{A}^{\circ}$.
Moreover, since $span\{ _\mc{A}\lb \xi, \eta \rb \mid \xi, \eta \in \mc{S}(\Gamma)\}$ is dense in $\mc{A}$ 
we see that the linear span of the short-time Fourier transforms $span\{V_{\eta}\xi \in L^1(G, \sigma)\}$ is dense in $L^1(G, \sigma)$. Thus so does $span\{ V^{\circ}_{\widehat{\eta}}\widehat{\xi} \mid \xi, \eta \in \mc{S}(\Gamma)\}$. Consequently, $\{V^{\circ}_{\eta'}\xi' \mid \eta', \xi' \in \mc{S}(\widehat{\Gamma})\}$ is dense in $L^1(G, \sigma)$ since all elements $\mc{S}(\widehat{\Gamma})$ comes from $\mc{S}(\Gamma)$ via Fourier transform. Other relations are easily checked; for instance, to check 
\[ _{\mc{A}^{\circ}}\lb \xi', \eta' \rb^*=  _{\mc{A}^{\circ}}\!\!\lb \eta', \xi' \rb\] we note that $_{\mc{A}^{\circ}}\lb \eta', \xi' \rb$ is $\pi^{\circ}\iota (V^{\circ}_{\eta'}\xi')$. Thus 
\[ _{\mc{A}^{\circ}}\!\lb \xi', \eta' \rb^*= \pi^{\circ}\iota ((V^{\circ}_{\eta'}\xi')^*)=\pi^{\circ}\iota (V^{\circ}_{\xi'}\eta')=_{\mc{A}^{\circ}}\!\!\lb \eta', \xi' \rb. \]
 \end{proof}
We note that there is a faithful trace $\Tr$ defined on both $\mc{A}$ and $\mc{A}^{\cc}$ given by 
\[\Tr (K) := k(0) \quad \text{ for $K=\pi (k)$ or $\pi^{\cc}\iota (k)$ where $k \in \mc{S}(G)$}.\] 
It satisfies $\Tr( k \natural l)=\Tr(l \natural k)$.\\ 

Now we consider a bona fide proper space of G.  Let $\Lambda$ be a  proper closed subgroup of $G=\widehat{\Gamma}\times \Gamma$. Once we fix the measures on $\Gamma$, $\widehat{\Gamma}$, and $\Lambda$, then the Haar measure $\mu_{G/\Lambda}$ on the quotiont group $G/\Lambda$  can be chosen that  for all $f\in L^1(G)$, 
\[\int_{G}f(g)dg=\int_{G/\Lambda}\!\int_{\Lambda}f(g+\lambda)d\mu_{\Lambda}(\lambda) d\dot{g} \quad \dot{g}=g+\Lambda \] holds which is called \emph{Weil's formula}.  With the uniquely determined measure 
$\mu_{G/\Lambda}$ we can define the size of $\Lambda$ or $\emph{covolume}$ of $\Lambda$, by $s(\Lambda)=\int_{G/\Lambda} 1 d\mu_{G/\Lambda}$. We are particularly interested in the case that $\Lambda$ is a discrete co-compact group or a lattice. In this case, $s(\Lambda) <  \infty$ is equal to the measure of any of its fundamental domains. The adjoint group of $\Lambda$ is the closed subgroup of $G$ given by 
\[ \Lambda^{\circ}=\{\lambda^{\circ}\mid \sigma_{sym}(\lambda^{\circ}, \lambda)=1 \quad \text{for all $\lambda \in \Lambda$} \}.\]

Then for $\Lambda$ a lattice we  consider two twisted Banach algebras $L^1(\Lambda, \sigma)$ and $L^1(\Lambda^{\circ}, \sigma^*)$ under the twisted convolution $\natural$ and the involution $^*$; for $\mf{a}=(a(\lambda)),\mf{b}=(b(\lambda)) \in L^1(\Lambda)$

\begin{equation}
\mathbf{a} \natural \mathbf{b}(\lambda)= \sum_{\nu \in \Lambda } a(\nu)b(\lambda-\nu)\sigma(\nu, \lambda-\nu)
\end{equation} 
and 
\begin{equation}
a^*(\lambda)=\sigma(\lambda, \lambda)\overline{a(-\lambda)} \quad \text{for $\lambda \in \Lambda$.}
\end{equation} 

Then $C^*(\Lambda, \sigma):=\mc{A}(\Lambda, \sigma)$ is the completion of $L^1(\Lambda, \sigma)$ under $\pi$. Similarly, $C^*(\Lambda^{\circ}, \sigma^*)=\mc{A}(\Lambda^{\cc}, \sigma^*)$ is the completion of $L^1(\Lambda^{\circ}, \sigma^*)$ under $\pi^{\circ}$ or $\pi^{*}$. We note that  for a co-compact $\Lambda$  the orthogonal measure $\mu_{\Lambda^{\circ}}$ 
satisfies 
\begin{equation}
\int_{\Lambda^{\circ}} a(\lambda^{\circ})d\mu_{\Lambda^{\circ}}(\lambda^{}\circ)=\frac{1}{s(\Lambda)}\sum_{\lambda^{\cc}\in \Lambda^{\cc}}a(\lambda^{\cc})
\end{equation} 
for $\mf{a} \in L^1(\Lambda^{\cc})$. 
Then we define a left action of $\mc{A}(\Lambda, \sigma)$ on $\mc{S}(\Gamma)$ as usual (see (\ref{E:action})) and a right action of $\mc{B}=\mc{A}(\Lambda^{\cc}, \sigma^*)$ as
\begin{equation}
\xi \cdot \mf{b}= \pi^*(\mf{b})(\xi)=\sum_{\lambda^{\cc}\in \Lambda^{\cc}} b(\lambda^{\cc})\pi^*(\lambda^{\cc})\xi
\end{equation}
for $\mf{b}\in L^1(\Lambda^{\circ})$ and a $\mc{B}$-valued hermitian product 
\begin{equation}
\lb \xi, \eta \rb_{\mc{B}}= \frac{1}{s(\Lambda)}\sum_{\lambda^{\cc} \in \Lambda^{\cc}}\lb \pi(\lambda^{\cc})\eta, \xi \rb \pi^*(\lambda^{\cc}). 
\end{equation}

Then $\mc{S}(\Gamma)$ still serves as an equivalence bimodule between $\mc{A}$ and $\mc{B}$ that are constructed by a lattice $\Lambda$ and the Schr\"{o}dinger representation. We note that a lattice or a proper subgroup $\Lambda$ plays a central role in signal analysis (see \cite{JL:Duality} for instance) and the following fact is fundamental in both operator algebras and Gabor analysis.
  
\begin{thm}(Rieffel)\label{T:Morita2}
 $\mc{S}(\Gamma)$ is an equivalence bimodule between $\mathcal{A}(\Lambda, \sigma)$ and $\mathcal{A}(\Lambda^{\circ}, \sigma^*)$. In particular, if we denote the quantization of $\mc{S}(\Lambda)$ under $\pi$ by $\mc{A}_{\infty}(\Lambda, \sigma)$ and the quantization of $\mc{S}(\Lambda^{\cc})$ under $\pi^*$ by $\mc{A}_{\infty}(\Lambda^{\circ}, \sigma^*) $ respectively, then still $\mc{S}(\Gamma)$ is an equivalence bimodule between $\mc{A}_{\infty}(\Lambda, \sigma)$ and $\mc{A}_{\infty}(\Lambda^{\circ}, \sigma^*)$. 
 \end{thm}
\begin{proof}
See \cite[Proposition 3.2]{R:Morita}.
\end{proof}

Using the dual Schr\"{o}dinger representation we can construct (Fourier dual) $C\sp*$-algebras $\mc{A}^{\cc}(\Lambda, \sigma)$ and $\mc{A}^{\circ}(\Lambda^{\circ}, \sigma^*)$ from $L^1(\Lambda, \sigma)$. 
Let $\mc{A}^{\cc}= \mc{A}^{
\cc}(\Lambda, \sigma)$ and $\mc{B}^{\cc}=\mc{A}^{\circ}(\Lambda^{\circ}, \sigma^*)$. Then on $\mc{S}(\widehat{\Gamma})$
a right action of $\mc{B}^{\cc}$ is given by
\begin{equation}
\xi' \cdot \mf{b}= (\pi^{\circ} \iota)^*(\mf{b})(\xi)=\sum_{\lambda^{\cc}\in \Lambda^{\cc}} b(\lambda^{\cc})(\pi^{\circ}\iota) ^*(\lambda^{\cc})\xi
\end{equation}
for $\mf{b} \in L^1(\Lambda^{\cc})$ and a $\mc{B}^{\cc}$-valued hermitian product 
\begin{equation}
\lb \xi', \eta' \rb_{\mc{B}}= \frac{1}{s(\Lambda)}\sum_{\lambda^{\cc} \in \Lambda^{\cc}}\lb \pi^{\cc}\iota (\lambda^{\cc})\eta', \xi' \rb (\pi^{\cc}\iota)^*(\lambda^{\cc}) 
\end{equation}

\begin{thm}\label{T:Morita3}
 $\mc{S}(\widehat{\Gamma})$ is an equivalence bimodule between $\mc{A}^{\cc}$ and $\mc{B}^{\cc}$. In particular, if we denote the quantization of $\mc{S}(\Lambda)$ under $\pi^{\cc}\iota$ by $\mc{A}_{\infty}^{\cc}$ and the quantization of $\mc{S}(\Lambda^{\cc})$ under $(\pi^{\cc}\iota)^*$ by $\mc{B}_{\infty}^{\cc}$ respectively, then still $\mc{S}(\Gamma)$ is an equivalence bimodule between $\mc{A}_{\infty}^{\circ}$ and $\mc{B}_{\infty}^{\cc}$.
 \end{thm}
\begin{proof}
Other relations are standard and can be proven easily following the proof of Theorem \ref{T:Morita1}\, or Section 2 in \cite{R:Morita}. We show the associativity relation 
\[ _{\mc{A}^{\cc}}\!\lb \xi', \eta' \rb \cdot  \zeta' = \xi' \cdot \lb \eta', \zeta' \rb_{\mc{B}^{\cc}} \quad \text{for $\xi', \eta', \zeta' \in \mc{S}(\widehat{\Gamma})$. } \] 
Since $\mc{F}:\mc{S}(\Gamma) \to \mc{S}(\widehat{\Gamma})$ is an isometric isomorphism, it is enough to show that for $\xi, \eta, \zeta \in \mc{S}(\Gamma)$
\[ _{\mc{A}^{\cc}}\!\lb \widehat{\xi}, \widehat{\eta} \rb \cdot  \widehat{\zeta} = \widehat{\xi} \cdot \lb \widehat{\eta}, \widehat{\zeta} \rb_{\mc{B}^{\cc}}. \]
Indeed, for $ \phi \in \mc{S}(\widehat{\Gamma})$
\[ 
\begin{split}
\lb \widehat{\xi} \cdot \lb \widehat{\eta}, \widehat{\zeta} \rb_{\mc{B}^{\cc}}, \widehat{\phi}\rb &=\frac{1}{s(\Lambda)}\sum_{\lambda^{\cc}\in \Lambda^{\cc}} \lb \pi^{\cc}\iota(\lambda^{\cc})\widehat{\zeta},  \widehat{\eta}\rb \lb \widehat{\xi}, \pi^{\cc}\iota(\lambda^{\cc}) \widehat{\phi} \rb \\ 
&=\frac{1}{s(\Lambda)}\sum_{\lambda^{\cc}\in \Lambda^{\cc}} \lb \mc{F} (\pi(\lambda^{\cc})\zeta),  \mc{F}(\eta)\rb \lb \mc{F}(\xi), \mc{F} (\pi (\lambda^{\cc})\phi) \rb \\ 
&=\frac{1}{s(\Lambda)}\sum_{\lambda^{\cc}\in \Lambda^{\cc}} \lb (\pi(\lambda^{\cc})\zeta,  \eta \rb \lb \xi, \pi (\lambda^{\cc})\phi \rb\\ 
&=\lb \xi \cdot \lb \eta, \zeta \rb_{\mc{B}}, \phi \rb\\
&=\lb _{\mc{A}}\lb\xi, \eta \rb \cdot \zeta  , \phi\rb\\
&=\lb \mc{F}( _{\mc{A}}\lb \xi, \eta \rb \cdot \zeta ), \mc{F}(\phi) \rb\\
&=\lb _{\mc{A}^{\cc}}\lb \widehat{\xi}, \widehat{\eta} \rb \cdot \widehat{\zeta} , \widehat{\phi} \rb
\end{split}
\]
\end{proof}

We also have faithful traces $\Tr_\mc{A}$ and $\Tr_{\mc{B}}$ defined on $\mc{A}(\mc{A}^{\cc})$ and $\mc{B}(\mc{B}^{\cc})$ respectively ; for $a\in L^1(\Lambda, \sigma)$ and $b\in L^1(\Lambda^{\cc}, \sigma*)$
  \[ \Tr_{\mc{A}}(a)=a(0), \quad \Tr_{\mc{B}}(b)=s(\Lambda)b(0). \]
\begin{prop}
The following equality holds;
\[\Tr_{\mc{A}}(_{\bullet}\lb \xi, \eta \rb)= \Tr_{\mc{B}}(\lb \eta, \xi \rb_{\bullet}).\]
\end{prop}  
\begin{proof}
It is evident that $\Tr_{\mc{A}}(_{\bullet}\lb \xi, \eta \rb)=\lb \xi, \eta \rb $ and 
\[  \Tr_{\mc{B}}(\lb \eta, \xi \rb_{\bullet})=s(\Lambda)  \frac{1}{s(\Lambda)} \lb \xi, \eta \rb=\lb \xi, \eta \rb.\]
\end{proof}

\section{Applications to solitons over noncommutative tori}\label{S:solitons}
From now on we assume that there is an infinitesimal action of  the 2-dimensional torus $\mb{T}^2$  on  both $\mc{A}$ and $\mc{A}^{\circ}$ which derives derivations $\partial_1$ and $\partial_2$ on both $\mc{A}$ and $\mc{A}^{\circ}$ (indeed, for $\Gamma=\mb{R}$  see \cite{L:Gabor}, or for $\Gamma=\mb{R}\times \Z_{q}$  see \cite{DJLL:Sigma}).  Also we assume that  the trace $\Tr_{\mc{A}} $ is invariant under the action which means that for any $K\in \mc{A}$ or $\mc{A}^{\cc}$
\begin{equation}
\Tr_{\mc{A}}(\partial_i K)=0, \quad j=1,2. 
\end{equation}
Moreover, we assume that on the equivalence bimodule $\mc{S}(\Gamma)$ and  $\mc{S}(\widehat{\Gamma})$ there is a connection via covariant derivatives $\nabla_1$ and $\nabla_2$ which commute with Fourier transform up to a scalar; 
\begin{equation}
\mc{F} \circ \nabla_{i}\equiv \nabla_{i+1}\circ \mc{F} \, (\text{mod 2})\, \text{for $i=1,2$}
\end{equation} 

We recall that the covariant derivatives satisfy the Leibniz rule and are compatible with the hermitian structure:
\begin{equation}
\nabla_j(K\cdot \xi)=(\partial_j K)\cdot \xi +K\cdot(\nabla_j \xi), \quad j=1,2
\end{equation}
\begin{equation}
\partial_j(_{\bullet}\lb \xi, \eta \rb)=_{\bullet}\!\!\lb \nabla_j \xi, \eta \rb+_{\bullet}\!\lb \xi, \nabla_j \eta \rb.
\end{equation}
On the other hand, the $L^2(\Gamma)(L^2(\widehat{\Gamma}))$-scalar product leads to the compatibility with respect to the right hermitian structure once we assume that ``integration by parts'' holds; 
\[ \int_{\Gamma} \nabla_j \xi (s)\overline{\eta}(s) ds= -\int_{\Gamma}\xi (s)\overline{\nabla_j \eta }(s) ds.\] 

If the right algebra is just $\C$, then the right Leibniz rule for the connection is automatic. Using the holomorphic structure $\overline{\partial}=\partial_1+i \partial_2$ and $\partial=\partial_1-i\partial_2$ we also consider the anti-holomorphic connection $\overline{\nabla}=\nabla_1+i\nabla_2$ and the holomorphic connection $\nabla=\nabla_1-i\nabla_2$ which are compatible with anti-holomorphic derivation and holomorphic derivation respectively. 

From the perspective of noncommutative analogues of non-linear (Bosonic) sigma models we are lead to consider the anti-self duality equation or the self duality equation for projections $p$ in $\mc{A}$ or $\mc{A}^{\circ}$ \cite{DKL:Sigma, DKL:Sigma2}; whenever $\mc{A}$, thereby $\mc{A}^{\cc}$, has sufficient projections, we can think of $*$-homomorphisms from the algebra of functions over a two-point set to a bona fide noncommutative space $\mc{A}$ or $\mc{A}^{\cc}$ and seek stable maps under the Polyakov type action functional. Such maps, in particular minimizing stable maps  correspond to the projections satisfying the either  the anti-self duality equation
\begin{equation}\label{E:antiselfdual}
(\partial p)p=0, 
\end{equation}
or the self duality equation
\begin{equation}\label{E:selfdual}
(\overline{\partial}p)p=0.
\end{equation}

In general, with in mind the discrete case,  if there is an equivalence bimodule $\mc{E}$ for the left algebra $\mc{A}$ and the right algebra $\mc{B}$, both of which are domains for the derivations  $\partial_i: \mc{A} \to \mc{A}$ and $\partial'_i: \mc{B} \to \mc{B}$ and connections $\nabla_j: \mc{E} \to \mc{E}$ on $\mc{E}$ which are compatible with both derivations, then the following is true. 
\begin{prop}\cite[Proposition 3.5]{DLL:Sigma}\label{P:solitons}
Let $\psi \in \mc{E}$ be such that $\lb \psi, \psi \rb_{\bullet}=1_{\mc{B}}$ with $p_{\psi}:=_{\bullet}\!\lb \psi, \psi \rb \in \mc{A}$ the corresponding projection. Let $\overline{\nabla}$ be the anti-holomorphic connection on $\mc{E}$. Then the projection $p_{\psi}$ is a solution of the self-duality equation of (\ref{E:selfdual}) or the anti-self duality equation of (\ref{E:antiselfdual}) respectively
\begin{equation}
\overline{\partial}(p_{\psi})p_{\psi}=0,\text{or}\quad  \partial(p_{\psi})p_{\psi}=0
\end{equation}
if and only if the $\psi$ is a generalized eigenvector of $\overline{\nabla}$ or $\nabla$ respectively, i.e. there exists $b\in \mc{B}$ such that 
\begin{equation}
\overline{\nabla}\psi=\psi \cdot b, \text{or}\quad \nabla(\psi)=\psi\cdot b.
\end{equation}
\end{prop}  

Rather the tight condition $\lb \psi, \psi \rb_{\mc{B}}=1_{\mc{B}}$ we can be relaxed by the invertibility of $\lb \psi, \psi \rb_{\mc{B}}$ through normalization . This strategy to find a projection in a left algebra through an Equivalence Bimodule  and a right algebra has been known since \cite{R:Morita, B:QT}.
 
\begin{prop}\label{P:Standard}\cite[Proposition 2.1]{Lee2:Sigma}
Let $\psi$ be an element such that $\lb \psi, \psi \rb_\mc{B}$ is invertible in $\mc{B}$. Then $_\mc{A}\langle \tilde{\psi}, \tilde{\psi} \rangle$ are projections in $A$ where $\tilde{\psi}=\psi \lb \psi, \psi\rb_B^{-1/2}$. We note that $\lb \psi, \psi \cdot \lb \psi, \psi \rb_{\mc{B}}^{-1} \rb_{\mc{B}}=\lb \psi \cdot \lb \psi, \psi \rb_{\mc{B}}^{-1}, \psi\rb_{\mc{B}} =1_\mc{B}$.
\end{prop}

Now we show that Fourier transform of $\psi$ in Heisenberg module $\mc{S}(\Gamma)$ satisfies anti-self duality equation if $\psi$ satisfies the self duality equation. 
\begin{prop}
Let $\psi \in \mc{S}(\Gamma)$ be such that $\| \psi\|^2=1$. Suppose that $p_{\psi} \in \mc{A}$ satisfies the self-duality equation. Then  $p_{\widehat{\psi}} \in \mc{A}^{\circ}$ satisfies the anti-self duality equation. 
\end{prop}  
\begin{proof}
Since $\| \widehat{\psi}\|^2=1$, $p_{\widehat{\psi}}$ is also a projection in $\mc{A}^{\circ}$. Since 
\[ (\nabla_1+i \nabla_2 )(\psi)=\psi \lambda \] for $\lambda \in \C$, 
\[\begin{split}
\nabla_1 \mc{F}(\psi) - i \nabla_2 \mc{F}(\psi)&=\mc{F} (\nabla_2 (\psi))- i \mc{F} (\nabla_1 (\psi))\\
&=-i \mc{F}(\nabla_1 \psi + i \nabla_2 \psi) \\ 
&=-i \mc{F}(\lambda \psi)=\lambda' \mc{F}(\psi).
\end{split}\]
Then the conclusion follows from Proposition \ref{P:solitons}.  
\end{proof} 

Like the continuous case we prove that following for a lattice $\Lambda$.   
\begin{thm}\label{T:ncsolitons}
Let $\mc{A}_{\infty}$, $\mc{A}_{\infty}^{\cc}$  and $\mc{B}_{\infty}, \mc{B}^{\cc}_{\infty}$ as in Theorem \ref{T:Morita2} and Theorem \ref{T:Morita3}. Suppose that $\lb \psi, \psi \rb_{\mc{B}}=1_\mc{B}$ and $\overline{\partial }p_{\psi} p_\psi =0$. Then $\lb \widehat{\psi}, \widehat{\psi} \rb_{\mc{B}^{\cc}}=1_{\mc{B}^{\cc}}$ and $\partial p_{\widehat{\psi}} p_{\widehat{\psi}}=0$.  In other words, $\psi \in \mc{S}(\Gamma)$ is a tight Gabor frame and satisfies the self duality equation, then $\widehat{\psi} \in \mc{S}(\widehat{\Gamma})$ is also a tight Gabor frame and satisfies the anti-self duality equation. 
\end{thm}
\begin{proof}
To show that $\lb \widehat{\psi}, \widehat{\psi} \rb_{\mc{B}^{\cc}}=1_{\mc{B}^{\cc}}$ 
we need to show that  for any $\phi \in \mc{S}(\Gamma)$ 
\[ \widehat{\phi} \cdot \lb \widehat{\psi}, \widehat{\psi} \rb_{\mc{B}^{\cc}}=\widehat{\phi}.\] 
Indeed, 
\[
\begin{split}
 \widehat{\phi} \cdot \lb \widehat{\psi}, \widehat{\psi} \rb_{\mc{B}^{\cc}}&=_{\mc{A}^{\cc}}\!\lb \widehat{\phi}, \widehat{\psi} \rb \cdot \widehat{\phi}\\
 &= \mc{F}(_{\mc{A}} \lb \phi, \psi \rb \cdot \psi)\\
 &=\mc{F}(\phi \cdot \lb \psi, \psi \rb_{\mc{B}})\\
 &=\widehat{\phi}.
\end{split}
\]
Suppose that $(\nabla_1 + i \nabla_2)\psi =\psi \cdot \lb \psi, \overline{\nabla}\psi \rb_{\mc{B}}$. 
Then \[
\begin{split}
\nabla \widehat{\psi}&=\nabla_1 (\mc{F} \psi) - i \nabla_2( \mc{F} \psi)\\
&= \mc{F} (\nabla_2 \psi)- i \mc{F} (\nabla_1 \psi) \\
&= -i \mc{F}(\nabla_1 + i \nabla_2)\psi \\
&=-i \mc{F}(\psi \cdot \lb \psi, \overline{\nabla}\psi \rb_{\mc{B}})\\
&=-i \mc{F}( _{\mc{A}}\lb \psi, \psi\rb \cdot \overline{\nabla} \psi )\\
&=-i _{\mc{A}^{\cc}}\lb \widehat{\psi},\widehat{\psi}\rb \cdot (\mc{F} \circ \overline{\nabla} \psi )\\
&=_{\mc{A}^{\cc}}\lb \widehat{\psi},\widehat{\psi}\rb \cdot (\nabla \widehat{\psi} )\\
&= \widehat{\psi}\cdot \lb \widehat{\psi}, \nabla (\widehat{\psi}) \rb_{\mc{B}^{\circ}}.
\end{split}
\]
\end{proof} 

Now let us comment the recent approach based on Gabor analysis to find the generalized eigenvector $\psi$ for $\overline{\nabla}$. In the continuous case, it reduces to find an ordinary eigenfunction or the first order ordinary PDE. However, in the discrete case, i.e.,  $\Lambda$ is a lattice of $\widehat{\Gamma}\times \Gamma$, it must satisfy 
\begin{equation}\label{E:eigen}
\overline{\nabla} \psi = \psi \cdot \lb \psi \cdot \lb \psi, \psi \rb^{-1}_{\mc{B}} , \overline {\nabla} \psi \rb_{\mc{B}}
\end{equation}
for $\psi$ provided that  $\lb \psi, \psi \rb_{\mc{B}}$ invertible \cite[Theorem4.5]{Lee2:Sigma}.  This could be achived if we can show that $\overline{\nabla}\psi$ belongs to the closed linear span of $\{\pi(\lambda^{\cc}) \psi \mid \lambda^{\cc} \in \Lambda^{\cc}\}$ because of the following theorem known as \emph{the duality principle in Gabor analysis}. 
\begin{thm}\cite[Theorem 6.5]{JL:Duality}
The following are equivalent. 
\begin{enumerate}
\item $\lb \psi, \psi \rb_{\mc{B}}$ invertible or the system $\{ \pi(\lambda)\psi\mid \lambda \in \Lambda\}$ generates a Gabor frame in term of Gabor analysis.  
\item $\{\pi(\lambda^{\cc}) \psi \mid \lambda^{\cc} \in \Lambda^{\cc}\} $ is a Riesz sequence  for $l^2(\Lambda^{\cc})$ i.e. for any sequence $\mf{c}=(c_{\lambda^{\cc}}) \in l^2(\Lambda^{\cc})$, there exist bounds $A, B$ such that 
\[ B\| \mf{c} \|^2 \le \|\frac{1}{s(\Lambda)} \sum_{\lambda^{\cc} \in \Lambda^{\cc}}c(\lambda^{\cc})\pi(\lambda^{\cc})\psi   \|^2 \le  A \| \mf{c}\|^2. \]
\end{enumerate}
\end{thm} 
\begin{rem}
It is a recent result that the duality principle holds for any closed subgroup $\Lambda$ not necessarily a lattice \cite{JL:Duality}.    
\end{rem}
\begin{prop}\label{P:Riesz}
If  $\{\pi(\lambda^{\cc}) \psi \mid \lambda^{\cc} \in \Lambda^{\cc}\} $ is a Riesz sequence  for $l^2(\Lambda^{\cc})$ and only if   $\{\pi^{\cc}\iota(\lambda^{\cc}) \widehat{\psi} \mid \lambda^{\cc} \in \Lambda^{\cc}\} $ is a Riesz sequence  for $l^2(\Lambda^{\cc})$.
\end{prop}
\begin{proof}
It is immediate from $\pi^{\cc}\iota (\lambda^{\cc}) \widehat{\psi}=\mc{F}(\pi(\lambda^{\cc})\psi)$ that 
for any for any sequence $\mf{c}=(c_{\lambda^{\cc}}) \in l^2(\Lambda^{\cc})$ \[
\|\frac{1}{s(\Lambda)} \sum_{\lambda^{\cc} \in \Lambda^{\cc}}c(\lambda^{\cc})\pi(\lambda^{\cc})\psi  \|^2
=\|\frac{1}{s(\Lambda)} \sum_{\lambda^{\cc} \in \Lambda^{\cc}}c(\lambda^{\cc})\pi^{\cc}\iota(\lambda^{\cc})\widehat{\psi}   \|^2.
\]
\end{proof} 
We point out that the above result is of independent interest in the sense of Gabor analysis; it says the Fourier transform of a Gabor frame becomes a frame using the dual Schr\"{o}dinger representation of a lattice $\Lambda$.  

A duality principle implies the following formula for $\xi \in W:=\overline{\Span}\{\pi(\lambda^{\cc}) \psi \mid \lambda^{\cc} \in \Lambda^{\cc}\} $
\begin{equation} \label{E:WRI}
\xi= \psi \cdot \lb \psi \cdot  \langle \psi , \psi \rangle_\mc{B}^{-1}, \xi \rb_\mc{B}.  
\end{equation}
We note that not all of $\mc{S}(\Gamma)$ belongs to $W$. Hence the following condition is important.

\begin{prop}\label{P:Riesz2}
Let $\psi \in \mc{S}(\Gamma)$. Then $\overline{\nabla} \psi$ belongs to $\overline{\Span}\{\pi(\lambda^{\cc}) \psi \mid \lambda^{\cc} \in \Lambda^{\cc}\} $ if and only if $\nabla (\widehat{\psi})$ belongs to $ \overline{\Span}\{\pi^{\cc}\iota(\lambda^{\cc}) \widehat{\psi} \mid \lambda^{\cc} \in \Lambda^{\cc}\}$. 
\end{prop}

If the assumption for $\psi$ is satisfied, then we take $\xi =\overline{\nabla} \psi$ in (\ref{E:WRI}) and 
obtain that 
\[ \overline{\nabla}\psi= \psi \cdot \lb \psi \cdot \lb \psi, \psi \rb^{-1}, \overline{\nabla}\psi \rb_\mc{B} \]
and  by Proposition \ref{P:Riesz} and by Proposition \ref{P:Riesz2}

\[ \nabla \widehat{\psi}= \widehat{\psi} \cdot  \lb \widehat{\psi} \cdot \lb \widehat{\psi}, \widehat{\psi} \rb_{\mc{B}^{\cc}}^{-1}, \nabla \widehat{\psi} \rb_{\mc{B}^{\cc}}. \] Thus $\widehat{\psi}$ generates a projection that is a solution of anti-self duality equation in $\mc{A}^{\circ}$. We summarize what we have observed so far:
\begin{thm}
Let $\mc{A}_{\infty}$, $\mc{A}_{\infty}^{\cc}$  and $\mc{B}_{\infty}, \mc{B}^{\cc}_{\infty}$ as in Theorem \ref{T:Morita2} and Theorem \ref{T:Morita3}. Then $\lb \psi, \psi \rb_{\mc{B}}$ invertible and $\overline{\nabla} \psi$ belongs to $\overline{\Span}\{\pi(\lambda^{\cc}) \psi \mid \lambda^{\cc} \in \Lambda^{\cc}\} $ if and only if $\lb \widehat{\psi}, \widehat{\psi} \rb_{\mc{B}^{\cc}}$ invertible and  $\nabla (\widehat{\psi})$ belongs to $ \overline{\Span}\{\pi^{\cc}\iota(\lambda^{\cc}) \widehat{\psi} \mid \lambda^{\cc} \in \Lambda^{\cc}\}$.  
\end{thm}
It is possible to obtain the solution of (\ref{E:eigen}) via other methods not using the duality principle. Thus the following is worth to observe.  
\begin{thm}\label{T:eigenvector}
Let $\psi \in \mc{S}(\Gamma)$. Then $\psi$ is the generalized eigenvector for $\overline{\nabla}$ if and only if $\widehat{\psi}$ is the generalized eigenvecor for $\nabla$. 
\end{thm}
\begin{proof}
Define an isomorphism $\mc{J}$ from $\mc{B}$ to $\mc{B}^{\cc}$  by 
\[\mc{J}(\sum_{\lambda^{\cc}} b(\lambda^{\cc})\pi^*(\lambda^{\cc})))=\sum_{\lambda^{\cc}} b(\lambda^{\cc}) (\pi^{\circ}\iota)^*(\lambda^{\cc}).\]
Then it follows that $\mc{J}(\lb \psi, \psi \rb_{\mc{B}})= \lb \widehat {\psi}, \widehat{\psi}\rb_{\mc{B}^{\cc}}$
By viewing $\mc{B}$ and $\mc{B}^{\cc}$ as an element of $\End_\mc{A}(\mc{E})$ and $\End_{\mc{A}^{\cc}}(\mc{E})$, we know that $\mc{J}(\lb \psi, \psi \rb_\mc{B}^{-1})=\lb \widehat{\psi}, \widehat{\psi} \rb_{\mc{B}^{\cc}}^{-1} $. Moreover, we claim that for $b\in \mc{B}$
\[ \mc{F}(\psi \cdot b)=\widehat{\psi}\cdot \mc{J}(b).\]
Indeed, 
\[
\begin{split}
\mc{F}(\psi \cdot b )&=\mc{F}( \sum_{\lambda^{\cc}}b(\lambda^{\cc})\pi^*(\lambda^{\cc}) \psi)\\
&=\sum_{\lambda^{\cc}}b(\lambda^{\cc})\mc{F} (\pi^{*}(\lambda^{\cc}) \psi )\\
&=\sum_{\lambda^{\cc}}b(\lambda^{\cc}) (\pi^{\cc}\iota)^*(\lambda^{\cc})\widehat{\psi}\\
&=\widehat{\psi} \cdot \mc{J}(b).  
\end{split}
\]
Then 
\[
\begin{split}
\mc{F} (\overline{\nabla} \psi)&=\mc{F}(\psi \cdot  \lb \psi \cdot \lb \psi, \psi \rb_{\mc{B}}^{-1}, \overline{\nabla}\psi \rb_{\mc{B}})\\
\nabla (\widehat{\psi})&= \widehat{\psi} \cdot \lb \mc{F}( \psi \cdot \lb \psi, \psi \rb_{\mc{B}}^{-1}), \nabla\widehat{\psi} \rb_{\mc{B}^{\cc}}\\
&=\widehat{\psi} \cdot \lb \widehat{\psi} \cdot \mc{J}(\lb \psi, \psi \rb_{\mc{B}}^{-1}), \nabla\widehat{\psi} \rb_{\mc{B}^{\cc}}\\
&=\widehat{\psi} \cdot \lb \widehat{\psi} \cdot \lb \widehat{\psi}, \widehat{\psi} \rb_{\mc{B}^{\cc}}^{-1}), \nabla\widehat{\psi} \rb_{\mc{B}^{\cc}}.
\end{split}
\]
\end{proof}
Further, sometimes we need to compute $\Tr_{\mc{B}}(b)$ for $b \in \mc{B}$ such that $\overline{\nabla}(\psi)=\psi \cdot b$ (see \cite{Lee2:Sigma, P:Ex}). In fact, $b=\lb \psi, \overline{\nabla}\psi \rb_{\bullet}$ for a tight frame $\psi$. It follows that $ \Tr_{\mc{B}}(b) = s(\Lambda)\lb \overline{\nabla}\psi, \psi \rb_{L^2(\Gamma)}$. Similarly, $\Tr_{\mc{B}^{\cc}}(\mc{J}(b))=s(\Lambda)\lb \nabla \widehat{\psi}, \widehat{\psi}\rb_{L^2(\widehat{\Gamma})}$.

\begin{cor}
Now suppose that $\Gamma \cong \widehat{\Gamma}$ and  a tight frame $\psi= \widehat{\psi}$ is a generalized eigenvector for $\overline{\nabla}$ for some $b\in \mc{B}$. Then $\Tr_{\mc{B}}(b)=0$.   
\end{cor}
\begin{proof}
Note that $\Tr_{\mc{B}}(b)=\Tr_{\mc{B}^{\cc}}(\mc{J}(b))$. Therefore, from Theorem \ref{T:eigenvector}
\[ \lb \nabla \widehat{\psi}, \widehat{\psi}\rb_{L^2(\widehat{\Gamma})}= \lb \overline{\nabla}\psi, \psi \rb_{L^2(\Gamma)}.\] 
Since $\widehat{\psi}=\psi$, \[ \lb \nabla \psi,\psi\rb_{L^2(\Gamma)}= \lb \overline{\nabla}\psi, \psi \rb_{L^2(\Gamma)}\] which implies that $\lb \nabla_2 \psi, \psi \rb_{L^2(\Gamma)}=0$. By  $\mc{F} \nabla_j \equiv \nabla_{j+1} \mc{F}$ mod $2$ for $j=1,2$.   
\[
\begin{split}
0&=\lb \nabla_2 \psi, \psi \rb_{L^2(\Gamma)}\\
&=\lb  \mc{F}(\nabla_2 \psi), \mc{F}(\psi) \rb_{L^2(\Gamma)}\\
&=\lb \nabla_1(\psi), \psi \rb_{L^2(\Gamma)}.
\end{split}
\]
\begin{rem}
The above statement is true for a general frame $\psi=\widehat{\psi}$. The proof is almost same with a careful touch.
\end{rem}
\end{proof}
Finally we prove that the topological charge of a soliton is also preserved under  Fourier transform up to sign change. We need to assume that our faithful traces $\Tr_{\mc{A}}$ and $\Tr_{\mc{B}}$ satisfy $\Tr_{\bullet}(\partial_j (\cdot))=0$. We recall that the topological charge of a projection $p$ in $\mc{A}$ is the first Connes-Chern number $c_1(p)$ given by
\[ c_1(p)=-\frac{1}{2\pi i} \Tr_{\mc{A}}(p(\partial_1 p \partial_2 p - \partial_2 p \partial_1 p)).\]

Then for the solitons $\psi $ such that $\lb \psi, \psi \rb_{\mc{B}}=1_{\mc{B}}$ we have 
\begin{equation}\label{E:charge}
c_1(p_{\psi})= -\frac{1}{2\pi i} \Tr_{\mc{B}} (\lb \psi, F_{12}\psi\rb_{\mc{B}})
\end{equation}
by \cite[Proposition 3.3]{DLL:Sigma}. Here $F_{12}:= \nabla_1\nabla_2- \nabla_2\nabla_1$ is the curvature of the covariant derivatives. 

\begin{prop}
For solitons $\psi \in \mc{S}(\Gamma)$ such that $\lb \psi, \psi \rb_{\mc{B}}=1_{\mc{B}}$, 
$c_i(p_{\psi})= -c_1(p_{\widehat{\psi}})$.  
\end{prop}
\begin{proof}
In view of (\ref{E:charge}) we need to look at the term 
\[ \begin{split}
F_{12} \widehat{\psi} &= \nabla_1\nabla_2 \widehat{\psi} -\nabla_2\nabla_2 \widehat{\psi}\\
&=\nabla_1 (\mc{F}(\nabla_1\psi))-\nabla_2 (\mc{F}(\nabla_2 \psi))\\
&=\mc{F}(\nabla_2 \nabla_2 \psi - \nabla_1\nabla_2 \psi)=-\mc{F}(F_{12} \psi).
\end{split}
 \]
Thus, \[
\begin{split}
c_1(p_{\widehat{\psi}})&= + \frac{1}{2\pi i}\Tr_\mc{B} (\lb \mc{F}(\psi)), \mc{F}(F_{12} \psi)\rb_{\mc{B}^{\cc}})\\
 &=  \frac{1}{2\pi i} \lb \mc{F}(\psi),  \mc{F}(F_{12} \psi)\rb_{L^2(\widehat{\Gamma})}=\frac{1}{2\pi i} \lb \psi, F_{12}\psi \rb_{L^2(\Gamma)}\\
 &=\frac{1}{2\pi i}\Tr_{\mc{B}}(\lb \psi, F_{12} \psi \rb_{\mc{B}})\\
 &=-c_1(p_{\psi}). 
\end{split} 
\] 
\end{proof}


\begin{bibdiv}

\begin{biblist}
\bib{CO:NCG}{article}{
   author={Connes, A.},
   title={$C\sp*$-alg\`{e}bres et g\'{e}ometrie diff\'{e}rentille},
   journal={C.R. Acad. Sci. Paris S\'{e}r. A},
   volume={290},
   number={13}
   date={1980},
   pages={599--604},
   issn={},
   review={MR1690050(81c:46053)},
   doi={},
}
\bib{B:QT}{article}{
   author={Boca, F.},
   title={Projections in rotation algebras and theta functions},
   journal={Comm. Math. Phys.},
   volume={202},
   date={1999},
   number={2}
   pages={325--357},
   issn={},
   review={MR1690050(2000j:46101)},
   doi={},
}
\bib{CR:YM}{article}{
   author={Connes, A.},
   author={Rieffel, M.}
   title={ Yang-Mills for noncommutative two-tori},
   journal={Contemp. Math.},
   volume={62},
   date={1987},
   pages={335--348},
   issn={0075-4102},
   review={\MR{454645 (56\#:12894)}},
   doi={},
}

\bib{DKL:Sigma}{article}{
   author={Dabrowski, L.},
   author={Krajewski, T.},
   author={Landi, G.},
   title={Some properties of Non-linear $\sigma$-models in noncommutative geometry},
   journal={Int. J. Mod. Phys.},
   volume={B14},
   date={2000},
   pages={2367--2382},
   review={\MR{0470685 (57 \#10431)}},
}

\bib{DKL:Sigma2}{article}{
   author={Dabrowski, L.},
   author={Krajewski, T.},
   author={Landi, G.},
   title={Non-linear $\sigma$-models in noncommutative geometry: fields with values in finite spaces},
   journal={Mod. Physics Lett. A},
   volume={18},
   date={2003},
   pages={2371--2379},
   review={},
}


\bib{DLL:Sigma}{article}{
   author={Dabrowski, L},
   author={Landi, G.},
   author={Luef, F}
   title={Sigma-model solitons on noncommutative spaces},
   journal={Lett. Math. Phys.},
   volume={105},
   date={2015},
   number={12},
   pages={1633--1688},
   issn={},
   review={\MR{3420593}},
   doi={10.1007/s11005-015-0790-x}
}

\bib{DJLL:Sigma}{article}{
  author={Dabrowski, L.},
  author={Jakobsen, M.},
  author={Landi, G.},
  author={Luef, F.} 
  title={Solitons of general topological charge over noncommutative tori}, 
  journal={arXiv:1801.08596},
   volume={},
  date={},  
number={}, 
pages={}, 
review={}
}

\bib{FS}{book}{
author={Feichtinger, H},
author={Strohmer},
title={Gabor analysis and Algorithms},
publisher={Springer Science+Business Media, LLC}
date={1998},
}
\bib{FL:Frame}{article}{
   author={Frank, M.},
   author={Larson, D.},
   title={Frames in Hilbert $C\sp*$-modules and $C\sp*$-algebras},
   journal={J. Operator Theory},
   volume={},
   date={2002},
   number={},
   pages={273--314},
   issn={},
   doi={},
}
\bib{GL:Gabor}{article}{
   author={Gr\"{o}chenig, K.},
   author={Lyubarskii, Y.}
   title={Gabor(super) frames with Hermite functions},
   journal={Math. Ann.},
   volume={345},
   date={2009},
   number={},
   pages={267--286},
   issn={},
   review={},
   doi={},
}
\bib{GS:Gabor}{article}{
   author={Gr\"{o}chenig, K.},
   author={St\"{o}ckler, J}
   title={Gabor frames and totally positive functions},
   journal={Duke  Math. J},
   volume={162},
   date={2013},
   number={5},
   pages={1003--1031},
   issn={},
   review={},
   doi={},
}

\bib{Howe}{article}
{author={Howe, R.}
title={On the role of the Heisenberg group in harmonic analysis},
 journal={Bull. Amer. Math. Soc.}
volume={3}, 
date={1980},
number={2},
pages={821--843},
 }

\bib{J:Segal}{article}
{author={Jacobsen, M.}
title={On a New Segal Algebra:A Review of the Feichtinger Algebra},
journal={J. Fourier Anal. Appl.}
volume={24},
date={2018},
pages={1579--1660}
}

\bib{JL:Duality}{article}
{author={Jakobsen, M.}
 author={Lemvig, J.}
title={Density and duality theorems for regular Gabor frames},
 journal={ J. Funct. Anal.}, 
 volume={270}
 date={2016}
 number={1}
 pages={229--263},
 }

\bib{Lee:Sigma}{article}{
   author={Lee, H.},
   title={A note on nonlinear $\sigma$-models in noncommutative geometry},
   journal={IDAQP},
   volume={19},
   date={2016},
   number={1},
   pages={},
   issn={0022-1236},
   review={\MR{2733573 (2011k:46079)}},
   doi={10.1142/S0239025716500065},
}
\bib{Lee2:Sigma}{article}{
author={Lee, H.},
   title={On a gauge action on sigma model solitons},
   journal={IDAQP},
   volume={21},
   date={2018},
   number={2},
   pages={},
   issn={0022-1236},
   review={},
   doi={10.1142/S023902571850008X},
}
\bib{L:Gabor}{article}{
   author={Luef, F.},
   title={Projections in noncommutative tori and Gabor frames},
   journal={Proc. A.M.S.},
   volume={139},
   date={2010},
   number={2},
   pages={571--582},
   issn={0002-9939},
   review={},
   doi={},
}
\bib{L:VBoverNT}{article}{
   author={Luef, F.},
   title={Projective modules over noncommutative tori are multi-window Gabor frames for modulation spaces},
   journal={J. Funct. Anal.},
   volume={257},
   date={2009},
   number={6},
   pages={1921--1946},
   issn={0379-4024},
   review={\MR{2540994}},
}
\bib{MR:Sigma}{article}{
   author={Mathai, V.},
  author={Rosenberg, J.}
   title={A noncommutative sigma-model},
   journal={J. Noncommut. Geom.},
   volume={5},
   date={2011},
   number={},
   pages={265--294},
   issn={},
   review={},
   doi={},
}
\bib{P:Ex}{article}{
   author={Polishchuck, P.},
   title={Analogues of the exponential map associated with complex structures on noncommutative two-tori},
   journal={Pacific J. Math.},
   volume={226},
   date={2006},
   number={1},
   pages={153--178},
   issn={},
   review={},
   doi={},
}
\bib{R:Morita}{article}{
   author={Rieffel, M.},
   title={Projective modules over higher-dimensional non-commutative tori },
   journal={Canad. J. Math.},
   volume={XL},
   date={1988},
   number={2},
   pages={257--338},
   review={},
   doi={},
   }

\end{biblist}

\end{bibdiv}
\end{document}